\documentclass[a4paper]{article}
\usepackage{titlesec}
\usepackage{amsfonts,amssymb,amsmath,indentfirst,amsthm}
\usepackage[sort&compress,numbers]{natbib}  
\usepackage[footnotesize]{caption}

\date{}

\begin{document}

\newtheorem{thm}{{\indent}Theorem}[section]
\newtheorem{cor}[thm]{{\indent}Corollary}
\newtheorem{lem}[thm]{{\indent}Lemma}
\newtheorem{prop}[thm]{{\indent}Proposition}
\newtheorem{result}[thm]{{\indent}Result}
\theoremstyle{definition}
\newtheorem{defn}[thm]{{\indent}Definition}
\newtheorem{rem}[thm]{{\indent}Remark}
\newtheorem{ex}[thm]{{\indent}Example}
\numberwithin{equation}{section}
\numberwithin{figure}{section}
\renewcommand{\figurename}{Fig.}
\newenvironment{keywords}{\par\textbf{keywords:}\mbox{  }}{ }
\newenvironment{Ack}{\par\textbf{Acknowledgements:}\mbox{  }}{ }
\newenvironment{coj}{\par\textbf{Conjecture:}\mbox{  }}{ }

\allowdisplaybreaks[4]

\title{Extremal problem of Hardy-Littlewood-Sobolev inequalities on compact Riemannian manifolds
\thanks{The project is supported by the National Natural Science Foundation of China (Grant No. 11571286, 11201443)  and Natural Science Foundation of Zhejiang Province(Grant
No. LY18A010013). }}
\author{Shutao Zhang\quad Yazhou Han\thanks{Corresponding author, Email: yazhou.han@gmail.com}\\
\footnotesize Department of Mathematics, College of
Science,\\[-0.15cm]
\footnotesize China Jiliang University, Hangzhou 310018, China}
\maketitle

\begin{abstract}
This paper studies  the existence of extremal problems for the Hardy-Littlewood-Sobolev inequalities on compact manifolds without boundary 
 via Concentration-Compactness principle.


\end{abstract}

\normalsize

\section{Introduction}\label{Section_1}

It is well known that classical Sobolev inequalities and Hardy-Littlewood-
Sobolev(HLS) inequalities are basic tools in analysis and geometry, and their sharp
constants play essential role on certain geometric and probabilistic
information. In fact, in past decades, these sharp inequalities were applied extensively in the study of curvature equations, see, e.g. \cite{CGY2002, GV2003, Br2003, LL2005, GLW2004, GV2007, DM2008} and references therein. Recently, there have been some interesting results concerning the globally defined fractional operators such as fractional Yamabe problems, fractional prescribing curvature problems, fractional Paneitz operators, etc. (see, e.g. \cite{GMS2012, GQ2011, GZ2003, JLX2011a, JLX2011b, JLX2014c, JX2011} and references therein), which are closely related to singular integral operator.  In particular,   the sharp HLS inequality  is immediately applied to discuss a class of prescribing integral curvature problems by Zhu \cite{Zhu2016} and integral equations on bounded domain  in \cite {DZ2018,DGZ2018}. So, HLS inequalities play essential role in the global analysis of some operators of geometric interest.

Motivated by these studies, there are some extensions of classical HLS inequalities, such as HLS inequality on the upper half space, HLS on compact manifolds, reversed HLS inequality, or HLS inequality on the Heisenberg group, see \cite{DZ1, DZ2, HZ2016, Frank-Lieb2012, DGZ2017, NN2017a, NN2017b} for details. This paper is mainly devoted to discuss the sharp HLS inequality on compact manifolds without boundary.

\medskip

Let $(M^n, g)$ be a given compact  Riemmanian manifold without boundary, $\alpha\in (0,n)$  be a parameter and $|x-y|_g$  represent the distance from $x$ to $y$ on $M^n$ under metric $g$. In \cite{HZ2016}, Han and Zhu have introduced the following integral operator
\begin{equation}\label{Operator-main-term}
    I_\alpha f(x) =\int_{M^n} \frac{f(y)} {|x-y|_g^{n-\alpha}} dV_y
\end{equation}
and got the following Hardy-Littlewood-Sobolev inequalities:
\begin{prop}[Proposition 1.1. in \cite{HZ2016}]\label{prop sharp ineq}
 Assume that $\alpha \in (0, n)$,  $1<p<\frac n\alpha$ and $q$ is given by \begin{equation}\label{critical exponent} \frac1q=\frac1p-\frac\alpha n,
 \end{equation} then  there is a positive constant $C(\alpha,p,M^n,g)$, such that
\begin{equation}\label{Roughly HLS M}
    ||I_{\alpha} f||_{L^q(M^n)} \le C(\alpha, p, M^n, g)||f||_{L^p(M^n)}
\end{equation}
holds for all  $f\in L^p({M}^n).$
Moreover, for $1\le r<q$,  operator $I_{\alpha} : L^p(M^n)  \to  L^r(M^n)$ is a compact  embedding.
\end{prop}

As is well known, it is important to study the extremal problems of \eqref{Roughly HLS M}, which can be stated as follows:
\begin{align}\label{extremal problem}
    N_{p,\alpha,M}:=&\sup\{\|I_\alpha f\|_{L^q(M^n)}:\ \|f\|_{L^p(M^n)}=1\} \nonumber\\
    :=&\sup \Bigl\{ \frac{\|I_\alpha f\|_{L^q(M^n)}} {\|f\|_{L^p(M^n)}}:\ f\in L^p(M^n) \backslash \{0\}\Bigr\}.
\end{align}
Equivalently, we can stated also as
\begin{align}
    N_{p,\alpha,M} :=& \sup \Bigl\{\Bigl| \int_{M^n}\int_{M^n} f(x) g(y) |x-y|_g^{\alpha-n} dV_x dV_y \Bigr|:\ \|f\|_p= \|g\|_t =1 \Bigr\} \label{defn extremal 2}\\
    :=&\sup_{\|f\|_p>0, \|g\|_t>0} \frac{\Bigl|\int_{M^n}\int_{M^n} f(x) g(y) |x-y|_g^{\alpha-n} dV_x dV_y\Bigr|} {\|f\|_p \|g\|_t},\label{defn extremal 3}
\end{align}
where $t=\frac q{q-1}$. 
In particular, we denote $N_{p,\alpha,\mathbb{R}^n}$ as $N_{p,\alpha}$.

In \cite{HZ2016}, Han and Zhu have discussed the extremal problems \eqref{extremal problem} for the conformal case, i.e. the case $p=t$ and $f=g$. Then as an application, they studied a class of integral curvature problems. Particularly, they give a new proof for the Yamabe problem on compact locally conformally flat manifold.

This paper will deal with the remaining cases. Firstly, we will get the following estimate to the sharp constant.
\begin{prop}[Estimate]\label{pro estimate}
$N_{p,\alpha,M}\geq N_{p,\alpha}$.
\end{prop}

Then, similar to the existence criteria of classical Yamabe problem, we will give the following the existence criteria of the extremal problems \eqref{extremal problem} by the Concentration-Compactness principle introduced by Lions (see \cite{Lions3, Lions4}). 
\begin{thm}[Criteria of Existence]\label{thm exist}
Under the assumption of Proposition \ref{prop sharp ineq} and if $N_{p,\alpha,M}>N_{p,\alpha}$, then the supremum is attained, i.e., there exists some function $f(x)\in L^p(M^n)$ such that $N_{p,\alpha,M}=\frac{\|I_\alpha f\|_{L^q(M^n)}} {\|f\|_{L^p(M^n)}}$.
\end{thm}

\begin{rem}
Let $G_x^g(y)=n(n-2)\omega_n\Gamma_x^g(y)$, where $\Gamma_x^g(y)$ is the Green's function with pole at $x$ for the conformal Laplacian operator $-\Delta_g+\frac{n-2}{4(n-1)}R_g$ and $\omega_n$ is the volume of the unit ball.
As discussed in \cite{HZ2016}, for the operator
    $$I_{M^n,g,\alpha}=\int_{M^n} \bigl[G_x^g(y)\bigr]^{\frac{\alpha-n}{2-n}} g(y) dV_y,$$
we can also get the similar results of estimate (Proposition \ref{pro estimate}) and existence criteria (Theorem \ref{thm exist}). Since the details of the proof is similar, so we omit it for conciseness.
\end{rem}

The plan of the paper is following. In Section \ref{Section_2}, we introduce some known facts and give a new proof of the compactness of operator \eqref{Operator-main-term} for convenience. Then, we present our Concentration-Compactness Lemma in the Section \ref{appendix CC-lemma}. Finally, Section \ref{Section_3} is devoted to get the estimate (Proposition \ref{pro estimate}) and prove the existence of extremal problem (Theorem \ref{thm exist}). 

\section{Preliminary}\label{Section_2}

Firstly, we recall the existence of the extremal problem of Classical Hardy-Littlewood-Sobolev inequalities on $\mathbb{R}^n$ as follows.
\begin{thm}[Theorem 2.3 of \cite{Lieb1983} \& Theorem 2.1 of \cite{Lions4}] There exist a pair of nonnegative functions $f\in L^p(\mathbb{R}^n)$ and $g\in L^t(\mathbb{R}^n)$ such that
\begin{equation}\label{Extremal Rn 1}\begin{cases}
    \int_{\mathbb{R}^n}|f|^p dx=\int_{\mathbb{R}^n}|g|^t dy=1\\ N_{p,\alpha}=\int_{\mathbb{R}^n}\int_{\mathbb{R}^n} f(x) g(y)|x-y|^{\alpha-n} dx dy.
\end{cases}\end{equation}
Hence, Extremal pair satisfies the Euler-Lagrange equation
\begin{equation}\label{Extremal Rn 2}
    \begin{cases}
        |x|^{\alpha-n}*g=N_{p,\alpha} f^{p-1}(x),\\
        |x|^{\alpha-n}*f=N_{p,\alpha} g^{t-1}(x).
    \end{cases}
\end{equation}
Furthermore, by scaling, we know that function pairs
\begin{equation}\label{Extremal Rn 3}
    f_\lambda(x)=\lambda^{-p/n}f(x/\lambda),\quad g_\lambda(y)=\lambda^{-t/n}g(y/\lambda),\quad \forall\lambda>0
\end{equation}
also satisfy \eqref{Extremal Rn 1} and \eqref{Extremal Rn 2}.
\end{thm}

For convenience, we introduce the following Young's inequality.
\begin{lem}[Young's inequality, Lemma 2.1 of \cite{HZ2016}]\label{lem Young} For a given compact manifold ($M^n, g)$,  define
$$g*h(x)=\int_{M^n} g(y) h(|y-x|_g) dV_y.$$
Then, there is a constant $C>0$, such that
$$||g*h||_{L^r} \le C||g||_{L^q} \cdot ||h||_{L^p},$$
where $p, q, r \in (1, \infty)$ and satisfy
$1+\frac 1r =\frac 1q+\frac 1p.$
\end{lem}

Following, we give a new proof of the compactness about the operator \eqref{Operator-main-term}.
\begin{prop}[Compactness]\label{pro compact}
For all $r\in [1,q)$, where $q$ is defined as \eqref{critical exponent}, operator $I_\alpha: L^p(M^n)\rightarrow L^r(M^n)$ is compact.
\end{prop}
\begin{proof}
Take any bounded sequence $\{f_m\}$ in $L^p(M^n)$. Then, there exists a subsequence (still denoted by $\{f_m\}$) and some function $f\in L^p(M)$ such that
\begin{equation}\label{comp 1}
    f_m\rightharpoonup f \quad\text{weakly in}\quad L^p(M^n).
\end{equation}
It is known that the proof will be completed if it holds that
    $$I_\alpha f_m \rightarrow I_\alpha f \quad \text{strongly in} \quad L^r(M^n).$$

Denoted by $K^\rho(t)=t^{\alpha-n}\chi_{\{t>\rho\}}$ and $K_\rho(t)=t^{\alpha-n}-K^\rho(t)$ for $t>0$, where $\rho>0$ is a parameter to be chosen later. Then, we decompose the integral operator as
    $$I_\alpha f_m(x)=K^\rho *f_m(x) + K_\rho *f_m(x) \triangleq I_\alpha^1 f_m(x) +I_\alpha^2 f_m(x).$$

Since, for any fixed $x\in M^n$, $K^\rho(|x-y|_g)\in L^{p'}(M^n)$ with respect to $y$, then weak convergence implies that $K^\rho *f_m\rightarrow K^\rho *f$ pointwisely. Notice also that
    $$\bigl| K^\rho * f_m(x) \bigr|\leq \|K^\rho\|_{P'} \|f_m\|_p\leq C(\rho),$$
where $C(\rho)$ is independent of $x$ and $m$. So, by dominated convergence theorem, we have that
    $$K^\rho *f_m \rightarrow K^\rho f \quad \text{strongly in} \quad L^r(M^n).$$

Since
    $$\int_{M^n} K_\rho(|x-y|_g)^s dV_x\leq C\rho^{(\alpha-n)s+n},$$
where $0<s<\frac n{n-\alpha}$, then we take parameter $s>1$ satisfying $\frac 1r+1=\frac 1p+\frac 1s$ and get from the Young's inequality (see Lemma \ref{lem Young}) that
    $$\|K_\rho * (f_m-f)\|_r\leq C\rho^{(\alpha-n)+n/s}\|f_m-f\|_p\leq C\rho^{(\alpha-n)+n/s}.$$

By now, through choosing first $\rho$ small and then $m$ large, we deduce the claimed convergence in $L^r(M^n)$. 
\end{proof}

Based on the Proposition \ref{pro compact}, we have the following conclusions.
\begin{rem}\label{rem converge}
For any bounded sequence $\{f_m\}\subset L^p(M^n)$, there exists a subsequence (still denoted by $\{f_m\}$) and some function $f\in L^p(M^n)$ such that
\begin{gather*}
    f_m\rightharpoonup f \quad \text{weakly in} \quad L^p(M^n),\\
    I_\alpha f_m\rightharpoonup I_\alpha f\quad \text{weakly in} \quad L^q(M^n),\\
    I_\alpha f_m\rightarrow I_\alpha f\quad \text{strongly in} \quad L^r(M^n)
\end{gather*}
for all $r\in [1,q)$. Furthermore, $I_\alpha f_m\rightarrow I_\alpha f$ pointwisely a.e. in $M^n$.
\end{rem}

\section{Concentration-Compactness Lemma}\label{appendix CC-lemma}

%
%
\begin{lem}\label{lem Concentrate-Compact}
Let $\{f_m\}\subset L^p(M^n)$ be a bounded nonnegative sequence and there exists some function $f\in L^p(M^n)$ such that
    $$f_m\rightharpoonup f \quad \text{weakly in} \quad L^p(M^n).$$
After passing to a subsequence, assume that $|I_\alpha f_m|^q dV_x$, $|f_m|^p dV_x$ converge weakly in the sense of measure to some bounded nonnegative measures $\nu,\ \mu$ on $M^n$. Then we have:

i) There exist some countable set $J$, a family $\{P_j: j\in J\}$ of distinct points in $M^n$, and a family $\{\nu_j: j\in J\}$ of nonnegative numbers such that
 \begin{equation}\label{formula CC 1}
    \nu=|I_\alpha f|^q dV_x+\sum_{j\in J} \nu_j\delta_{P_j},
 \end{equation}
where $\delta_{P_j}$ are the Dirac-mass of mass $1$ concentrated at $P_j\in M^n$;

ii)In addition we have
\begin{equation}\label{formula CC 2}
    \mu\geq |f|^pdV_x+\sum_{j\in J} \mu_j\delta_{P_j}
\end{equation}
for some family $\{\mu_j>0: j\in J\}$, where $\mu_j$ satisfy
\begin{equation}\label{formula CC 3}
    \nu_j^{1/q}\leq N_{p,\alpha}\mu_j^{1/p} \quad \text{for all} \quad j\in J.
\end{equation}
In particular, $\sum_{j\in J}\nu_j^{p/q}<+\infty$.
\end{lem}

{\bf Proof of i).}
By the conditions of the sequence $\{f_m\}\subset L^p(M^n)$, we know from the Remark \ref{rem converge} that
\begin{gather*}
    I_\alpha f_m\rightharpoonup I_\alpha f\quad \text{weakly in} \quad L^q(M^n),\\
    I_\alpha f_m\rightarrow I_\alpha f\quad \text{strongly in} \quad L^r(M^n)\\
    I_\alpha f_m\rightarrow I_\alpha f \quad \text{pointwisely a.e. in} \quad M^n,
\end{gather*}
where $r\in [1,q)$. Then, Br\'{e}zis-Lieb Lemma leads that
\begin{align*}
    0=& \lim_{m\rightarrow +\infty} \int_{M^n} \left( |I_\alpha f_m|^q -|I_\alpha(f_m-f)|^q -|I_\alpha f|^q \right) dV_x\\
    =& \int_{M^n} d\nu -\int_{M^n} |I_\alpha f|^q dV_x -\lim_{m\rightarrow+\infty} |I_\alpha (f_m-f)|^q dV_x.
\end{align*}
So, it is sufficient to discuss the case $f\equiv 0$. By the classical argument of Lions (see \cite{Lions3, Lions4}), it is sufficient to prove
\begin{equation}\label{formula CC 4}
    \left(\int_{M^n} |\varphi|^q d\nu \right)^{1/q} \leq N_{p,\alpha,M} \left(\int_{M^n} |\varphi|^p d\mu \right)^{1/p}, \quad \forall \varphi\in C_0^\infty(M^n).
\end{equation}

Since, for any $\varphi(x)\in C_0^\infty(M^n)$,
\begin{align*}
    &\left( \int_{M^n} |\varphi(x) I_\alpha f_m|^q dV_x \right)^{1/q}\\
    \leq & \left( \int_{M^n} |I_\alpha (\varphi f_m)|^q dV_x\right)^{1/q} + \left( \int_{M^n} |\varphi(x) I_\alpha f_m -I_\alpha(\varphi f_m)|^q dV_x \right)^{1/q}\\
    \leq & N_{p,\alpha,M} \left( \int_{M^n} |\varphi f_m|^p dV_x\right)^{1/p}  + \left( \int_{M^n} |\varphi(x) I_\alpha f_m -I_\alpha(\varphi f_m)|^q dV_x \right)^{1/q},
\end{align*}
then we get as $m\rightarrow +\infty$ that
\begin{align*}
    \left( \int_{M^n} |\varphi|^q d\nu \right)^{1/q} \leq & N_{p,\alpha,M} \left( \int_{M^n} |\varphi|^p d\mu \right)^{1/p}\\ &+ \lim_{m\rightarrow +\infty} \left( \int_{M^n} |\varphi(x) I_\alpha f_m -I_\alpha(\varphi f_m)|^q dV_x \right)^{1/q}.
\end{align*}
So, we can obtain \eqref{formula CC 4} if
\begin{equation}\label{formula CC 5}
    \lim_{m\rightarrow +\infty} \left( \int_{M^n} |\varphi(x) I_\alpha f_m -I_\alpha(\varphi f_m)|^q dV_x \right)^{1/q}=0.
\end{equation}

Notice that
\begin{align*}
    |\varphi(x) I_\alpha f_m -I_\alpha(\varphi f_m)| = & \Bigl|\int_{M^n} (\varphi(x)-\varphi(y)) |x-y|_g^{\alpha-n} f_m(y) dV_y\Bigr|\\
    \leq & C \int_{M^n} |x-y|_g^{\alpha+1-n} |f_m(y)| dV_y
\end{align*}
and
    $$R(x,y):=(\varphi(x)-\varphi(y)) |x-y|_g^{\alpha-n} \in L^r(M^n),$$
where $r\leq +\infty$ if $\alpha+1-n\geq 0$ and $r<\frac n{n-\alpha-1}$ if $\alpha+1-n<0$. If $\alpha+1-n\geq 0$, we can prove \eqref{formula CC 5} by dominated convergence theorem. While for the case $\alpha+1-n<0$, we obtain through the Hardy-Littlewood-Sobolev inequalities \eqref{Roughly HLS M} that
\begin{equation*}
    \int_{M^n} R(x,y) f_m(y) dV_y \in L^s(M^n),
\end{equation*}
where $s=(\frac 1p-\frac{\alpha+1}n)^{-1}>q$. Furthermore, repeating the proof of Proposition \ref{pro compact}, we have
    $$\int_{M^n} R(x,y) f_m(y) dV_y \rightarrow \int_{M^n} R(x,y) f(y) dV_y=0 \quad \text{strongly in} \quad L^q(M^n).$$
So, we get \eqref{formula CC 5} and complete the proof of i). \hfill$\Box$

\medskip

{\bf Proof of ii).}
Since
    $$f_m\rightharpoonup f \quad \text{weakly in} \quad L^p(M^n),$$
then, $\mu\geq |f|^p dV_x$. So, we just have to show that for each fixed $j\in J$,
    $$\nu_j^{1/q}=\nu(\{P_j\})^{1/q}\leq N_{p,\alpha} \mu(\{P_j\})^{1/p}=N_{p,\alpha}\mu_j^{1/p}.$$

For point $P_j\in M^n$, choose a neighbourhood $\Omega_{P_j}\subset M^n$ so that for $\delta>0$ small enough, in a normal coordinate,  $\exp(B_{\delta}) \subset \Omega_{P_j}$ and
    $$(1-\epsilon)I\leq g(x)\leq(1+\epsilon)I,\quad \forall x\in B_{\delta}.$$
Take $\varphi_\lambda(x)=\varphi(\frac x\lambda)$, where $\varphi(x)\in C_0^\infty(\mathbb{R}^n)$ satisfies $0\leq\varphi(x)\leq 1,\ \varphi(0)=1,\ \text{supp}\ \varphi\subset B_1$ and $\lambda\in (0,\delta)$. Then,
\begin{align*}
    I_\alpha((\varphi_\lambda\circ \exp^{-1}) \cdot f_m) =& \int_{M^n} (\varphi_\lambda\circ\exp^{-1})(y) f_m(y) |x-y|_g^{\alpha-n} dV_g(y)\\
    =&\int_{B_\delta} \varphi_\lambda(y) (f_m\circ\exp)(y) |x-y|_g^{\alpha-n} \sqrt{\det g(y)} dy\\
    \leq & \frac{(1+\epsilon)^{n/2}} {(1-\epsilon)^{n-\alpha}} \int_{B_\delta} \varphi_\lambda(y) (f_m\circ\exp)(y) |x-y|^{\alpha-n} dy
\end{align*}
and
\begin{align*}
    &\left(\int_{\exp(B_\delta)} |I_\alpha ((\varphi_\lambda\circ\exp^{-1})\cdot f_m)|^q dV_x \right)^{1/q}\\
    \leq & (1+\epsilon)^{n/(2q)} \left( \int_{B_\delta} |I_\alpha ((\varphi_\lambda\circ\exp^{-1})\cdot f_m)|^q dx\right)^{1/q}\\
    \leq &\frac{(1+\epsilon)^{\frac n2(1+\frac 1q)}} {(1-\epsilon)^{n-\alpha}} \left( \int_{B_\delta} \Bigl| \int_{B_\delta} \varphi_\lambda(y) (f_m\circ\exp)(y) |x-y|^{\alpha-n} dy\Bigr|^q dx\right)^{1/q}\\
    \leq &\frac{(1+\epsilon)^{\frac n2(1+\frac 1q)}} {(1-\epsilon)^{n-\alpha}} N_{p,\alpha} \left( \int_{B_\delta} |\varphi_\lambda(y) (f_m\circ\exp)(y)|^p dy \right)^{1/p}\\
    \leq & \frac{(1+\epsilon)^{\frac n2(1+\frac 1q)}} {(1-\epsilon)^{\frac n{2p}+n-\alpha}} N_{p,\alpha} \left( \int_{\exp(B_\delta)} |(\varphi_\lambda\circ\exp^{-1}) \cdot f_m|^p dV_y \right)^{1/p}.
\end{align*}
So,
\begin{align}\label{formula CC 6}
    &\left( \int_{M^n} |(\varphi_\lambda\circ\exp^{-1}) \cdot I_\alpha f_m|^q dV_x\right)^{1/q}\nonumber\\
    \leq& \left(\int_{\exp(B_\delta)} |I_\alpha ((\varphi_\lambda\circ\exp^{-1})\cdot f_m)|^q dV_x \right)^{1/q}\nonumber\\
    &\quad +\left( \int_{\exp(B_\delta)} |(\varphi_\lambda\circ\exp^{-1}) \cdot I_\alpha f_m -I_\alpha((\varphi_\lambda\circ\exp^{-1}) \cdot f_m)|^q dV_x\right)^{1/q}\nonumber\\
    \leq & \frac{(1+\epsilon)^{\frac n2(1+\frac 1q)}} {(1-\epsilon)^{\frac n{2p}+n-\alpha}} N_{p,\alpha} \left( \int_{\exp(B_\delta)} |(\varphi_\lambda\circ\exp^{-1}) \cdot f_m|^p dV_y \right)^{1/p}+\textbf{I},
\end{align}
where
    $$\textbf{I}:=\left( \int_{\exp(B_\delta)} |(\varphi_\lambda\circ\exp^{-1}) \cdot I_\alpha f_m -I_\alpha((\varphi_\lambda\circ\exp^{-1}) \cdot f_m)|^q dV_x\right)^{1/q}.$$
Repeating the argument of \eqref{formula CC 5}, we have, as $m\rightarrow +\infty$,
    $$\textbf{I}\rightarrow \left( \int_{\exp(B_\delta)} |(\varphi_\lambda\circ\exp^{-1}) \cdot I_\alpha f -I_\alpha((\varphi_\lambda\circ\exp^{-1}) \cdot f)|^q dV_x\right)^{1/q}.$$
So, letting $m\rightarrow +\infty$ leads
\begin{align}\label{formula CC 7}
    &\left( \int_{M^n} |\varphi_\lambda\circ\exp^{-1}|^q d\nu \right)^{1/q}\nonumber\\
    \leq & \frac{(1+\epsilon)^{\frac n2(1+\frac 1q)}} {(1-\epsilon)^{\frac n{2p}+n-\alpha}} N_{p,\alpha} \left( \int_{M^n} |(\varphi_\lambda\circ\exp^{-1}) |^p d\mu \right)^{1/p}\nonumber\\
    &\quad +\left(\int_{\exp(B_\delta)} |(\varphi_\lambda\circ\exp^{-1}) \cdot I_\alpha f -I_\alpha((\varphi_\lambda\circ\exp^{-1}) \cdot f)|^q dV_x\right)^{1/q}.
\end{align}
Since
\begin{align*}
    &\int_{\exp(B_\delta)} |(\varphi_\lambda\circ\exp^{-1}) \cdot I_\alpha f|^q dV_x\rightarrow 0\quad \text{as}\quad \lambda\rightarrow 0^+
\intertext{and}
    &\left(\int_{\exp(B_\delta)} |I_\alpha((\varphi_\lambda\circ\exp^{-1}) \cdot f)|^q dV_x\right)^{1/q}\\
    \leq & C\left( \int_{B_\delta} |(\varphi_\lambda\circ\exp^{-1}) \cdot f|^p dV_y\right)^{1/p}\rightarrow 0\quad \text{as}\quad \lambda\rightarrow 0^+,
\end{align*}
then we can complete the proof by letting $\lambda\rightarrow 0^+$ and $\epsilon\rightarrow 0^+$.
\hfill$\Box$

\section{Estimate and criteria of existence \label{Section_3}}

{\bf Proof of Proposition \ref{pro estimate}.}
For small positive constant $\lambda>0$, recall that  $f_\lambda(x)$ and $g_\lambda(y)$ are given in \eqref{Extremal Rn 3}.
Take
    $$\tilde{f}(x)=\begin{cases} f_\lambda(x), &\text{ in } B_\delta(0)\\ 0, &\text{ in } \mathbb{R}^n\backslash B_\delta(0) \end{cases} \quad \text{and} \quad \tilde{g}(y)=\begin{cases} g_\lambda(y), &\text{ in } B_\delta(0)\\  0, &\text{ in } \mathbb{R}^n\backslash B_\delta(0) \end{cases}$$
where $\delta>0$ is a fixed constant to be determined later. Then, for small enough $\lambda$ and by \eqref{Extremal Rn 2},
\begin{align}\label{est-1}
    &\int_{\mathbb{R}^n}\int_{\mathbb{R}^n} \tilde{f}(x) \tilde{g}(y) |x-y|^{\alpha-n}dxdy \nonumber\\
    =&\int_{\mathbb{R}^n}\int_{\mathbb{R}^n} f_\lambda(x) g_\lambda(y) |x-y|^{\alpha-n}dxdy \nonumber\\ &-\int_{|x|>\delta}\int_{\mathbb{R}^n} f_\lambda(x) g_\lambda(y) |x-y|^{\alpha-n}dxdy \nonumber\\
    &-\int_{\mathbb{R}^n}\int_{|y|>\delta} f_\lambda(x) g_\lambda(y) |x-y|^{\alpha-n}dxdy \nonumber\\
    &+\int_{|x|>\delta}\int_{|y|>\delta} f_\lambda(x) g_\lambda(y) |x-y|^{\alpha-n}dxdy \nonumber\\
    =&N_{p,\alpha}-N_{p,\alpha}\int_{|x|>\delta} f_\lambda^p(x) dx-N_{p,\alpha}\int_{|y|>\delta} g_\lambda^t(y) dy\nonumber\\
    &+\int_{|x|>\delta}\int_{|y|>\delta} f_\lambda(x) g_\lambda(y) |x-y|^{\alpha-n}dxdy \nonumber\\
    := &N_{p,\alpha} -\textbf{I} -\textbf{II} +\textbf{III},
\end{align}
where, for fixed $\delta>0$ and as $\lambda\rightarrow 0^+$,
\begin{align*}
    &\textbf{I}:=N_{p,\alpha}\int_{|x|>\delta} f_\lambda^p(x) dx=N_{p,\alpha}\int_{|x|>\delta/\lambda} f^p(x) dx\rightarrow 0,\\
    &\textbf{II}:=N_{p,\alpha}\int_{|y|>\delta} g_\lambda^t(y) dy\rightarrow 0,\\
    &\textbf{III}:=\int_{|x|>\delta}\int_{|y|>\delta} f_\lambda(x) g_\lambda(y) |x-y|^{\alpha-n}dxdy\\
    &\hspace{0.5cm}\leq C \left(\int_{|x|>\delta} f_\lambda^p dx\right)^{1/p} \left(\int_{|y|>\delta} g_\lambda^t dx\right)^{1/t}\rightarrow 0.
\end{align*}
So, for small enough  $\lambda$,
\begin{align}\label{est-2}
    &\frac{ \int_{\mathbb{R}^n} \int_{\mathbb{R}^n} \tilde{f}(x) \tilde{g}(y) |x-y|^{\alpha-n}dxdy} {\|\tilde{f}\|_{L^p(\mathbb{R}^n)} \|\tilde{g}\|_{L^t(\mathbb{R}^n)}} \nonumber\\
    \geq &\frac{N_{p,\alpha}-\textbf{I}-\textbf{II}} {\|f_\lambda\|_{L^p(\mathbb{R}^n)} \|g_\lambda\|_{L^t(\mathbb{R}^n)}} =N_{p,\alpha}-\textbf{I}-\textbf{II}.
\end{align}

For any given point $P\in M^n$, choose a neighbourhood $\Omega_P\subset M^n$ so that for $\delta>0$ small enough, in a normal coordinate,  $\exp(B_{\delta})\subset\Omega_P$ and
$$(1-\epsilon)I\leq g(x)\leq(1+\epsilon)I,\quad \forall x\in B_{\delta}.$$
Thus,
$$(1-\epsilon)|x-y|\leq|x-y|_g\leq(1+\epsilon)|x-y|,  \ \  \quad \forall x, \ y \in B_{\delta}.$$

In the normal coordinates with respect to the center $P\in M^n$, let
\begin{align*}
    u(x)=\begin{cases} f_\lambda(\exp^{-1}(x)), & \text{in}\quad \exp(B_{\delta})\\ 0, &\text{in}\quad M^n\backslash\exp(B_{\delta}) \end{cases}
\intertext{and}
     v(y)=\begin{cases} g_\lambda(\exp^{-1}(y)), & \text{in}\quad \exp(B_{\delta})\\ 0, &\text{in}\quad M^n\backslash\exp(B_{\delta}). \end{cases}
\end{align*}
Then
\begin{align}
    &\int_{M^n}|u|^p dV_x \leq (1+\epsilon)^{\frac n2} \int_{B_\delta(0)} |f_\lambda(x)|^p dx, \nonumber\\
    &\int_{M^n}|v|^t dV_y \leq (1+\epsilon)^{\frac n2} \int_{B_\delta(0)} |g_\lambda(y)|^t dy, \nonumber\\
    &\int_{M^n}\int_{M^n} u(x)v(y) |x-y|_g^{\alpha-n} dV_xdV_y\nonumber\\
    &\quad =\int_{B_{\delta}(0)} \int_{B_{\delta}(0)} \frac{u(x)v(y)} {|x-y|_g^{n-\alpha}} \sqrt{\det g(x)} \sqrt{\det g(y)} dxdy \nonumber\\
    &\quad \geq \int_{B_{\delta}(0)} \int_{B_{\delta}(0)} \frac{f_\lambda(x) g_\lambda(y)} {(1+\epsilon)^{n-\alpha} |x-y|^{n-\alpha}} (1-\epsilon)^n dxdy \nonumber\\
    &\quad =\frac{(1-\epsilon)^n} {(1+\epsilon)^{n-\alpha}} \int_{B_{\delta}(0)} \int_{B_{\delta}(0)} \frac{f_\lambda(x)g_\lambda(y)} {|x-y|^{n-\alpha}}dxdy.
 \label{est-3}
\end{align}
Thus
\begin{align*}
    N_{p,\alpha,M}&\geq \frac{\int_{M^n} \int_{M^n} {u(x)v(y)} |x-y|_g^{\alpha-n} dV_xdV_y} {\|u\|_{L^p(M^n)} \|v\|_{L^t(M^n)}}\\
    &\geq \frac{ \frac{(1-\epsilon)^{n}} {(1+\epsilon)^{n-\alpha}} \int_{B_{\delta}(0)} \int_{B_{\delta}(0)} f_\lambda(x) g_\lambda(y) |x-y|^{\alpha-n} dxdy} {(1+\epsilon)^{\frac n2 (\frac 1p+\frac 1t)} \|f_\lambda\|_{L^p(B_\delta(0))} \|g_\lambda\|_{L^t(B_\delta(0))}} \\
    &\geq \frac{(1-\epsilon)^n} {(1+\epsilon)^{\frac n2(\frac 1p+\frac 1t)+n-\alpha}} \left(N_{p,\alpha} -\textbf{I} -\textbf{II}\right).
\end{align*}
Sending  $\epsilon$ and  $\lambda$ to $0$, we obtain the estimate. \hfill$\Box$

\bigskip

{\bf Prof of Theorem \ref{thm exist}.}
Take a maximizing nonnegative sequence $\{f_m(x)\}\subset L^p(M^n)$ satisfying $\int_{M^n}f_m^p dV_x=1$ and
\begin{equation}\label{formula exist 1}
    \|I_\alpha f_m\|_{L^q(M^n)}\rightarrow N_{p,\alpha,M}, \quad\text{as}\quad m\rightarrow +\infty.
\end{equation}
Then, there exist a subsequence of $\{f_m\}$ (still denoted by $\{f_m\}$) and some function $f\in L^p(M^n)$ such that
    $$f_m \rightharpoonup f \quad \text{weakly in}\quad L^p(M^n).$$
Because of the Hardy-Littlewood-Sobolev inequalities \eqref{Roughly HLS M}, we know that
\begin{equation}\label{formula exist 2}
    \mu_m=|f_m|^p dV_x,\quad \nu_m=|I_\alpha f_m|^q dV_x
\end{equation}
are two families of bounded measures. So, there exist two nonnegative bounded measures $\mu$ and $\nu$ on $M^n$ such that
    $$\mu_m\rightharpoonup \mu, \nu_m\rightharpoonup \nu$$
weakly in the sense of measures.

Applying the Concentration-Compactness Lemma (see Lemma \ref{lem Concentrate-Compact}), we have
\begin{equation}\label{formula exist 3}
    \nu=|I_\alpha f|^q dV_x+\sum_{j\in J} \nu_j\delta_{P_j}, \quad \mu\geq |f|^pdV_x+\sum_{j\in J} \mu_j\delta_{P_j},
\end{equation}
and $\nu_j^{1/q}\leq N_{p,\alpha}\mu_j^{1/p}$ for all $j\in J$. Since $\int_{M^n}d\mu =\lim_{m\rightarrow+\infty} \int_{M^n}|f_m|^p dV_x =1,$ then $\int_{M^n}|f|^p dV_x\leq 1$ and $\mu_j\leq 1,\ j\in J$.

We claim that $\mu_j=0,\ j\in J$, which deduce that $\nu_j=0,\ j\in J$.

In fact, otherwise, combining \eqref{formula exist 3} and the fact $\frac qp>1$, we have
\begin{align}\label{formula exist 4}
    N_{p,\alpha,M}^q&=\lim_{m\rightarrow+\infty} \int_{M^n} |I_\alpha f_m|^q dV_x =\int_{M^n} d\nu\nonumber\\
    &=\int_{M^n} |I_\alpha f|^q dV_x +\sum_{j\in J} \nu_j \nonumber\\
    &\leq N_{p,\alpha,M}^q \|f\|_{L^p(M^n)}^q +\sum_{j\in J} N_{p,\alpha}^q\mu_j^{q/p} \nonumber\\
    &< N_{p,\alpha,M}^q \left(\int_{M^n} |f|^p dV_x\right)^{q/p} +\sum_{j\in J} N_{p,\alpha,M}^q\mu_j^{q/p} \nonumber\\
    &\leq N_{p,\alpha,M}^q\left(\int_{M^n} |f|^p dV_x+\sum_{j\in J}\mu_j\right)^{q/p}\nonumber\\
    &=N_{p,\alpha,M}^q \left(\int_{M^n}d\mu\right)^{q/p} =N_{p,\alpha,M}^q,
\end{align}
which is a contradiction.

Repeating the process of \eqref{formula exist 4}, we have that
    $$N_{p,\alpha,M}^q=\int_{M^n}|I_\alpha f|^q dV_x \quad \text{and} \quad \int_{M^n}|f|^p dV_x=1,$$
i.e., $f$ is a maximizer.
\hfill$\Box$

\begin{appendix}
\titleformat{\section}{\Large\bfseries}{Appendix\ \thesection} {1em}{}

\end{appendix}



\end{document}